\documentclass[12pt]{amsart}
\usepackage{amscd,amsmath,amsthm,amssymb}
\usepackage{amsfonts,amssymb,amscd,amsmath,enumerate,verbatim}
\usepackage{tikz, float} \usetikzlibrary {positioning}
\usepackage[left]{lineno}
\usepackage{pstcol,pst-plot,pst-3d}
\usepackage{color}
\usepackage{pstricks}
\usepackage{stmaryrd}
\usepackage[utf8]{inputenc}
\usepackage{pstricks-add}
\usepackage{graphicx}
\newpsstyle{fatline}{linewidth=1.5pt}
\newpsstyle{fyp}{fillstyle=solid,fillcolor=verylight}
\definecolor{verylight}{gray}{0.97}
\definecolor{light}{gray}{0.9}
\definecolor{medium}{gray}{0.85}
\definecolor{dark}{gray}{0.6}

 \def\G{{\mathcal G}}

  \def\Mc{{\mathcal M}}

 \def\opn#1#2{\def#1{\operatorname{#2}}} 
 \opn\chara{char} \opn\length{\ell} \opn\pd{pd} \opn\rk{rk}
 \opn\projdim{proj\,dim} \opn\injdim{inj\,dim} \opn\rank{rank}
 \opn\depth{depth} \opn\grade{grade} \opn\height{height}
 \opn\embdim{emb\,dim} \opn\codim{codim}
 
 \opn\Tr{Tr} \opn\bigrank{big\,rank}
 \opn\superheight{superheight}\opn\lcm{lcm}
 \opn\trdeg{tr\,deg}
 \opn\reg{reg} \opn\lreg{lreg} \opn\ini{in} \opn\lpd{lpd}
 \opn\size{size} \opn\sdepth{sdepth}
 \opn\link{link}\opn\fdepth{fdepth}\opn\lex{lex}
 \opn\tr{tr}
 \opn\type{type}
 \opn\gap{gap}
 \opn\arithdeg{arith-deg}
 \opn\revlex{revlex}
 \opn\div{div} \opn\Div{Div} \opn\cl{cl} \opn\Cl{Cl}
 \opn\Spec{Spec} \opn\Supp{Supp} \opn\supp{supp} \opn\Sing{Sing}
 \opn\Ass{Ass} \opn\Min{Min}\opn\Mon{Mon}
 \opn\Ann{Ann} \opn\Rad{Rad} \opn\Soc{Soc}
 \opn\Im{Im} \opn\Ker{Ker} \opn\Coker{Coker} \opn\Am{Am}
 \opn\Hom{Hom} \opn\Tor{Tor} \opn\Ext{Ext} \opn\End{End}
 \opn\Aut{Aut} \opn\id{id}
 
 \opn\nat{nat}
 \opn\pff{pf}
 \opn\Pf{Pf} \opn\GL{GL} \opn\SL{SL} \opn\mod{mod} \opn\ord{ord}
 \opn\Gin{Gin} \opn\Hilb{Hilb}\opn\sort{sort}
 \opn\PF{PF}\opn\Ap{Ap}
 \opn\mult{mult}
 \opn\bight{bight}
 \opn\aff{aff}
 \opn\relint{relint} \opn\st{st}
 \opn\lk{lk} \opn\cn{cn} \opn\core{core} \opn\vol{vol}  \opn\inp{inp} \opn\nilpot{nilpot}
 \opn\link{link} \opn\star{star}\opn\lex{lex}\opn\set{set}
 \opn\width{wd}
 \opn\Fr{F}
 \opn\QF{QF}
 \opn\G{G}
 \opn\type{type}\opn\res{res}
 \opn\conv{conv}
 \opn\Deg{Deg}
 \opn\Sym{Sym}
 \opn\gr{gr}
 
 \def\pot#1#2{#1[\kern-0.28ex[#2]\kern-0.28ex]}

 \opn\dirlim{\underrightarrow{\lim}}
 \opn\inivlim{\underleftarrow{\lim}}
 \let\to=\rightarrow
 
 \def\Implies{\ifmmode\Longrightarrow \else
         \unskip${}\Longrightarrow{}$\ignorespaces\fi}
 \def\implies{\ifmmode\Rightarrow \else
         \unskip${}\Rightarrow{}$\ignorespaces\fi}
 \def\iff{\ifmmode\Longleftrightarrow \else
         \unskip${}\Longleftrightarrow{}$\ignorespaces\fi}

 \let\:=\colon
 \newtheorem{Theorem}{Theorem}[section]
 \newtheorem{Lemma}[Theorem]{Lemma}
 \newtheorem{Corollary}[Theorem]{Corollary}
 
 \newtheorem{Remark}[Theorem]{Remark}
 
 \newtheorem{Example}[Theorem]{Example}

 \let\epsilon\varepsilon
 \let\kappa=\varkappa
 \def\qed{\ifhmode\textqed\fi
       \ifmmode\ifinner\quad\qedsymbol\else\dispqed\fi\fi}
 \def\textqed{\unskip\nobreak\penalty50
        \hskip2em\hbox{}\nobreak\hfil\qedsymbol
        \parfillskip=0pt \finalhyphendemerits=0}
 \def\dispqed{\rlap{\qquad\qedsymbol}}
 
 \opn\dis{dis}
 \def\pnt{{\raise0.5mm\hbox{\large\bf.}}}
 
 \opn\Lex{Lex}

\begin{document}

\title{Toric rings arising from vertex cover ideals}

\author {J\"urgen Herzog, Takayuki Hibi  and  Somayeh Moradi }

\address{J\"urgen Herzog, Fachbereich Mathematik, Universit\"at Duisburg-Essen, Campus Essen, 45117
Essen, Germany} \email{juergen.herzog@uni-essen.de}

\address{Takayuki Hibi, Department of Pure and Applied Mathematics,
Graduate School of Information Science and Technology, Osaka
University, Suita, Osaka 565-0871, Japan}
\email{hibi@math.sci.osaka-u.ac.jp}

\address{Somayeh Moradi, Department of Mathematics, School of Science, Ilam University,
P.O.Box 69315-516, Ilam, Iran}
\email{so.moradi@ilam.ac.ir}

\dedicatory{ }

\begin{abstract}
We extend the sortability concept to monomial ideals which are not necessarily generated in one degree and as an application we obtain normal Cohen-Macaulay toric rings attached to vertex cover ideals of graphs. Moreover, we consider a construction on a graph called a clique multi-whiskering which always produces vertex cover ideals with componentwise linear powers.
\end{abstract}

\maketitle

\setcounter{tocdepth}{1}

\section*{Introduction}
Let $S=K[x_1,\ldots,x_n]$ be the polynomial ring in $n$ variables over a field $K$.
A sortable set of monomials in $S$ was first considered in~\cite{Sturmfels}, for a set of monomials of the same degree. This is an interesting property in the study of toric rings. In fact the defining ideal of a toric ring generated by a sortable set of monomials has a quadratic Gr\"{o}bner basis which is formed by sorting relations. This in particular implies that such toric rings are Koszul and normal Cohen-Macaulay domains.

In this paper we consider the concept of sortability more generally for a set of monomials in arbitrary degrees and study the toric rings of the form $A=K[u_1t,\ldots,u_mt]$, where $\{u_1,\ldots,u_m\}$ is a sortable set of monomials. Using a criterion by Sturmfels~\cite{Sturmfels} we show that in this more general case, the sorting relations form a quadratic Gr\"{o}bner basis for the defining ideal of $A$, as well.  This allows us to find Koszul and normal Cohen-Macaulay toric rings, whose generators are in arbitrary degrees. Such toric rings appear when we consider vertex cover ideals of graphs and the generators of the toric ring correspond to the minimal vertex covers of a graph.
In particular, when $G$ is a proper interval graph (also known as closed graphs), then the vertex cover ideal of $G$ is sortable and hence the toric ring attached to the minimal vertex covers of $G$ possess the above mentioned properties (see Theorem~\ref{Properinterval}).

A toric ring attached to an ideal $I$, which plays an important role in the study of powers of $I$, is the Rees ring $\mathcal{R}(I)$. In this regard, the so-called $x$-condition which is a criterion on the defining ideal of a Rees ring, is a strong tool to check whether an ideal has linear or componentwise linear powers.
For any graded ideal generated in one degree, the $x$-condition implies that all powers of $I$ have linear resolutions (see ~\cite{HHZ}).
In~\cite{HHM} a natural extension of the $x$-condition was given to arbitrary graded ideals and modules which under some additional assumptions guarantees that all powers of the graded ideal are componentwise linear. This was shown to be the case for vertex cover ideals of some families of graphs, namely biclique graphs, path graphs and Cameron-Walker graphs, whose bipartite part is a complete bipartite graph.

Few families of graphs have been known so that their vertex cover ideals  have componentwise linear powers.
Cohen-Macaulay chordal graph~\cite{HHO}, trees~\cite{Kumar}, $(C_4, 2K_2)$-free graphs~\cite{E}, star graphs based on $K_n$~\cite{HHO} and generalized star graphs~\cite{Moh} are such families. In this paper we attach to an arbitrary graph $G$ a family of graphs called clique multi-whiskerings of $G$. Considering the Rees rings of their vertex cover ideals and using the $x$-condition we show that all powers of vertex cover ideals of such graphs are componentwise linear (see Lemma~\ref{x condition} and Theorem~\ref{powers}).

\section{Toric rings attached to sortable sets of monomials}

In this section we study the Gr\"{o}bner basis of toric rings generated by sortable sets of monomials.

Throughout this section $S=K[x_{1},\ldots, x_{n}]$ is a polynomial ring over a field $K$. Let $u$ and $v$ be two monomials in $S$. Write $uv=x_{i_1}x_{i_2}\cdots x_{i_{t}}$ with $i_1\leq i_2\leq \cdots\leq i_{t}$, and
set $u^\prime=x_{i_1}x_{i_3}\cdots x_{i_{t'}}$ and $ v^\prime =x_{i_{2}}x_{i_4}\cdots x_{i_{t''}}$, where $t'$ and $t''$ are the biggest odd and even number with $t'\leq t$ and $t''\leq t$, respectively. The pair
$(u^\prime,v^\prime)$ is called the \emph{sorting} of $(u,v)$ and is denoted by $\sort(u,v).$
For example $\sort(x_1x_2x_4x_5,x_3)=(x_1x_3x_5,x_2x_4)$.

The pair $(u,v)$ is called a {\em sorted pair},  if $\sort(u,v)=(u,v)$.  Otherwise,  $(u,v)$ is called an {\em unsorted pair}. Observe that if $\sort(u,v)=(u',v')$, then  $uv=u'v'$ and we have either $\deg u'=\deg v'$ or $\deg u'=\deg v'+1$. Moreover, if $u$ and $v$ are squarefree, then $u'$ and $v'$ are squarefree, as well.
A subset $\Mc\subset \Mon(S)$ is called \emph{sortable} if $\sort(u,v)\in \Mc\times \Mc$ for all $(u,v)\in \Mc\times \Mc.$

\medskip

A $q$-tuple $(u_1,\ldots,u_q)$ of monomials in $S$ is called {\em sorted}, if $(u_i,u_j)$ is a sorted pair for any $i<j$.
If $(u_1,\ldots,u_q)$ is sorted, then for some $0\leq r<q$, $\deg(u_1)=\deg(u_2)=\cdots=\deg(u_r)=d+1$, and $\deg(u_i)=d$ for any $r<i\leq q$. Then we say that $(u_1,\ldots,u_q)$ is {\em sorted of type} $(d,r)$.

For each $i$ assume that $u_i=x_{\ell_{1,i}}x_{\ell_{2,i}}\cdots x_{\ell_{d_i, i}}$, where $\ell_{1,i}\leq \cdots \leq \ell_{d_i,i}$.
Then $(u_1,\ldots,u_q)$ is sorted of type $(d,r)$ if and only if $d_1=\cdots=d_r=d+1$, $d_{r+1}=\cdots=d_q=d$, and

\begin{eqnarray*}
\ell_{1,1}\leq \ell_{1,2}\leq \cdots\leq \ell_{1,q}\leq \ell_{2,1}\leq \cdots \leq\ell_{2,q}\leq \cdots \leq \\
\ell_{d,1}\leq \cdots \leq \ell_{d,q}
& \leq \ell_{d+1,1} \leq\cdots \leq \ell_{d+1,r}.
\end{eqnarray*}

\begin{Lemma}\label{qsorted}
Let $(u_1,\ldots,u_q)$ be an arbitrary $q$-tuple of monomials in $S^q$ and let $m=\sum_{i=1}^q \deg(u_i)$. Then
there exists a unique sorted $q$-tuple $(u'_1,\ldots,u'_q)$ such that $u_1\cdots u_q=u'_1\cdots u'_q$. Moreover, it is of type $(d,r)$, where $d>0$ and $0\leq r<q$ are the unique integers with  $m=qd+r$.
\end{Lemma}

\begin{proof}
For the desired sorted $q$-tuple $(u'_1,\ldots,u'_q)$, we must have $0 \leq \deg(u'_i)-\deg(u'_j)\leq 1$ for $i<j$, since $(u'_i,u'_j)$ should be sorted.
So if $\deg(u'_q)=d$, then $d\leq \deg(u'_i)\leq d+1$ for all $i$. This together with the equality $\sum_{i=1}^q \deg(u'_i)=qd+r$, imply that $\deg(u'_i)=d+1$ for $1\leq i\leq r$, and $\deg(u'_i)=d$ for $r+1\leq i\leq q$. Now, write $u_1\cdots u_q$ as $\prod_{i=1}^{m} x_{k_i}$ with $k_1\leq k_2\leq \cdots\leq k_m$. Then the unique sorted $q$-tuple is obtained by setting $u'_i=\prod_{s=0}^{d} x_{k_{sq+i}}$ for $1\leq i\leq r$ and $u'_i=\prod_{s=0}^{d-1} x_{k_{sq+i}}$ for $r+1\leq i\leq q$.
\end{proof}

We denote the unique sorted $q$-tuple $(u'_1,\ldots,u'_q)$ in Lemma~\ref{qsorted} by $\sort(u_1,\ldots,u_q)$. Observe that for any permutation $\sigma$ on $[q]$, we have $$\sort(u_1,\ldots,u_q)=\sort(u_{\sigma(1)},\ldots,u_{\sigma(q)}).$$

Let $B$ be a sortable set of monomials in $S$. We are interested in the Gr\"{o}bner basis of the toric ring $K[B]=K[u:\ u\in B]$. Let $T=K[y_u:\ u\in B]$ be the polynomial ring with the order of indeterminates given by $y_u>y_v$ if $u>_{\lex}v$.
Let $$\varphi:T\to K[B]$$ be the $K$-algebra homomorphism defined by $\varphi(y_u)=u$ for any $u\in B$ and let $P_B$ be
the kernel of $\varphi$. We show that there exists a suitable monomial order on $T$ called the sorting order, so that $P_B$ possesses a quadratic Gr\"{o}bner basis. To this aim we need to prove the following

\begin{Lemma}\label{sortfinitely}
With the assumptions of Lemma~\ref{qsorted}, the unique sorted $q$-tuple can be obtained from $(u_1,\ldots,u_q)$ in a finite number of single sorting steps.
\end{Lemma}

\begin{proof}
Let $\textbf{u}=(u_1,\ldots,u_q)$ and assume that $d_i=\deg(u_i)$. Then $m=\sum_{i=1}^q d_i=qd+r$. First we claim that after finitely many single sortings on $\textbf{u}$ we obtain $\textbf{v}=(v_1,\ldots,v_q)$ such that $|\deg(v_i)-\deg(v_j)|\leq 1$ for all $i$ and $j$.

If $|d_i-d_j|\leq 1$ for all $i$ and $j$, then there is nothing to prove and $v_i=u_i$ for all $i$. Assume that there exists $i$ and $j$ with $|d_i-d_j|>1$.
Let $\mathcal{S}$ be the set of all $q$-tuples of monomials and consider the function $f:\mathcal{S}\to \mathbb{Z}$ with $$f(w_1,\ldots,w_q)=\sum_{i<j} |\deg(w_i)-\deg(w_j)|.$$
Clearly $f\geq 0$. We show that if $d_s=\min\{d_i:\ 1\leq i\leq q\}$, $d_t=\max\{d_i:\ 1\leq i\leq q\}$ and $\textbf{u}'$ is a $q$-tuple obtained from
$\textbf{u}$ by sorting $(u_t,u_s)$, then $f(\textbf{u}')<f(\textbf{u})$. Once this is shown, since $f$ is bounded below, we can deduce that after finitely many steps we get a $q$-tuple $\textbf{v}$ with $|\deg(v_i)-\deg(v_j)|\leq 1$ for all $i$ and $j$.

Without loss of generality we may assume that $d_1\geq \cdots\geq d_q$.
Then $t=1$, $s=q$ and $d_1-d_q>1$.

Let $\sort(u_1,u_q)=(u',u'')$, $d'=\deg(u')$ and $d''=\deg(u'')$. Then $0\leq d'-d''\leq 1$ and $d'+d''=d_1+d_q$.

We have $$f(\textbf{u})=\sum_{i<j} (d_i-d_j)=\sum_{i,j\notin\{1,q\}, i<j} (d_i-d_j)+\sum_{s=2}^q (d_1-d_s) +\sum_{s=2}^q (d_s-d_q)$$ and
$$f(\textbf{u}')=\sum_{i,j\notin\{1,q\},i<j} (d_i-d_j)+\sum_{s=2}^{q-1} (|d'-d_s|+|d''-d_s|)+(d'-d'').$$
Thus we need to show that $\sum_{s=2}^{q-1} (|d'-d_s| +|d''-d_s|)+1<(q-1)(d_1-d_q)$.

\medskip
Fix an integer $2\leq s\leq q$. If $d_s\leq d''$, then $|d'-d_s| +|d''-d_s|=(d'+d'')-2d_s=d_1+d_q-2d_s\leq d_1-d_q$, since $d_s\geq d_q$. If
$d_s>d''$, then $d_s\geq d'$. So $|d'-d_s| +|d''-d_s|=2d_s-(d'+d'')=2d_s-(d_1+d_q)\leq d_1-d_q$, since $d_s\leq d_1$. Therefore,
$\sum_{s=2}^{q-1} (|d'-d_s| +|d''-d_s|)+1\leq (q-2)(d_1-d_q)+1<(q-1)(d_1-d_q)$. So the claim is proved.

By what shown above, in the sequel we may reduce to the case of a $q$-tuple $\textbf{u}$ with $\deg(u_1)=\cdots=\deg(u_r)=d+1$ and $\deg(u_i)=d$ for $r+1\leq i\leq q$.
Let $\mathcal{M}$ be the set of all $q$-tuples of monomials with this property.
Let $u_i=x_{s_{i,1}}\cdots x_{s_{i,d+1}}$ with $s_{i,1}\leq \cdots\leq s_{i,d+1}$ for $1\leq i\leq r$ and $u_i=x_{s_{i,1}}\cdots x_{s_{i,d}}$ with $s_{i,1}\leq \cdots\leq s_{i,d}$ for $r+1\leq i\leq q$. Write
\begin{equation}\label{s_i,j}
u_1\cdots u_q=(x_{s_{1,1}}x_{s_{2,1}}\cdots x_{s_{q,1}})(x_{s_{1,2}}x_{s_{2,2}}\cdots x_{s_{q,2}})\cdots (x_{s_{1,d}}\cdots x_{s_{q,d}})(x_{s_{1,d+1}}\cdots x_{s_{r,d+1}}).
\end{equation}
For each $i$ and $j$ set $k_{i+(j-1)q}=s_{i,j}$. Then equation (\ref{s_i,j}) is
\begin{equation}\label{k_i}
u_1\cdots u_q=(x_{k_1}\ldots x_{k_q})(x_{k_{q+1}}\cdots x_{k_{2q}})\cdots (x_{k_{dq+1}}\cdots x_{k_{dq+r}}).
\end{equation}
Define the function $g:\mathcal{M}\to \mathbb{Z}$ as $$g(\textbf{u})=\sum_{i<j} (k_j-k_i).$$
Then $g$ is bounded above by a suitable multiple of $n-1$.
So we only need to show that if for some $i<j$, $(u_i,u_j)$ is an unsorted pair with $\sort(u_i,u_j)=(u'_i,u'_j)$ and $\textbf{u}'$ is the $q$-tuple obtained from $\textbf{u}$ by this single sorting, then $g(\textbf{u})<g(\textbf{u}')$. Let $u_iu_j=x_{k_{i_1}}\cdots x_{k_{i_{t}}}$, where $t\in [2d,2d+2]$ and $i_1\leq \cdots\leq i_t$. We set $\mathcal{T}=\{i_1,\ldots,i_t\}$. Then
$$g(\textbf{u})=\sum_{i,j\in \mathcal{T},i<j}(k_j-k_i)+\sum_{i,j\notin \mathcal{T},i<j}(k_j-k_i)+\sum_{j\notin \mathcal{T}}(\sum_{i<j,i\in \mathcal{T}}(k_j-k_i)+\sum_{i>j,i\in \mathcal{T}}(k_i-k_j)).$$
Let $\textbf{k}':k'_1,k'_2,\ldots,k'_{dq+r}$ be the corresponding sequence as in (\ref{k_i}) for $\textbf{u}'$. Since $d_i,d_j\in\{d,d+1\}$ and $d_i\geq d_j$, we have $\deg(u'_i)=d_i$ and $\deg(u'_j)=d_j$. Thus the set of indices defining the corresponding subsequence for $(u'_i,u'_j)$ in $\textbf{k}'$ is $\mathcal{T}$ as well. Hence for any $k_{\ell}\notin \mathcal{T}$, $k_{\ell}=k'_{\ell}$ which implies that $\sum_{i,j\notin \mathcal{T},i<j}(k_j-k_i)=\sum_{i,j\notin \mathcal{T},i<j}(k'_j-k'_i)$.
Also since $k'_{i_1}\leq \cdots\leq k'_{i_t}$, we have $$\sum_{i,j\in \mathcal{T},i<j}(k_j-k_i)<\sum_{i,j\in \mathcal{T},i<j}|k_j-k_i|=\sum_{i,j\in \mathcal{T},i<j}(k'_j-k'_i).$$
The inequality is strict, because $(u_i,u_j)$ is unsorted.
To complete the proof it is enough to show that for any $j\notin \mathcal{T}$,
$$\sum_{i<j,i\in \mathcal{T}}(k_j-k_i)+\sum_{i>j,i\in \mathcal{T}}(k_i-k_j)\leq\sum_{i<j,i\in \mathcal{T}}(k'_j-k'_i)+\sum_{i>j,i\in \mathcal{T}}(k'_i-k'_j),$$
Since $k_j=k'_j$, this is equivalent to
\begin{equation}\label{thirdsum}
\sum_{i>j,i\in \mathcal{T}}k_i-\sum_{i<j,i\in \mathcal{T}}k_i\leq\sum_{i>j,i\in \mathcal{T}}k'_i-\sum_{i<j,i\in \mathcal{T}}k'_i.
\end{equation}
Since $k'_{i_1}\leq \cdots\leq k'_{i_t}$ and $\{k_{i_1},\ldots,k_{i_t}\}=\{k'_{i_1},\ldots,k'_{i_t}\}$, we have $\sum\limits_{i<j,i\in \mathcal{T}}k'_i\leq \sum\limits_{i<j,i\in \mathcal{T}}k_i$ and $\sum_{i\in \mathcal{T}}k_i=\sum_{i\in \mathcal{T}}k'_i$. Therefore, $\sum\limits_{i>j,i\in \mathcal{T}}k'_i\geq \sum\limits_{i>j,i\in \mathcal{T}}k_i$. From these inequalities we can conclude (\ref{thirdsum}).
\end{proof}

From Lemma~\ref{sortfinitely} and \cite[Theorem 3.12]{Sturmfels} we conclude

\begin{Theorem}
Let $B$ be a sortable subset of monomials of $S$ and
$$F=\{y_uy_v-y_{u'}y_{v'}:\  u, v\in B,\  (u, v)\textrm{ unsorted},\ (u',v')=\sort(u,v)\}.$$
Then there exists a monomial order $<$ on $T=K[y_u:\ u\in B]$ which is called the sorting
order such that for every $f_{u,v} =y_uy_v-y_{u'}y_{v'}\in F$, $\ini_<(f_{u,v}) =y_uy_v$.
\end{Theorem}

\begin{proof}
Consider $F$ as a family of marked binomials $f_{u,v}=\underline{y_uy_v}-y_{u'}y_{v'}$. Lemma~\ref{sortfinitely} implies that every sequence of reductions modulo $F$ terminates in a finite number of steps.
Then by~\cite[Theorem 3.12]{Sturmfels}, $F$ is marked coherently, which means that there exists a monomial order $<$ on $T$ such
that $\ini_<(f_{u,v})=y_uy_v$ for every unsorted pair $(u,v)$.
\end{proof}

\begin{Theorem}\label{Gbasis}
Let $B=\{u_1,\ldots,u_m\}$ be a sortable set of monomials and $A=K[u_1t,\ldots,u_mt]$ be the $K$-algebra attached to $B$ with the toric ideal $P_B\subset T=K[y_u:\ u\in B]$. Then the set
$$\mathcal{G}=\{y_uy_v-y_{u'}y_{v'}:\ (u, v)\textrm{ unsorted},\ (u',v')=\sort(u,v)\}$$
is the reduced Gr\"{o}bner basis of $P_B$ with respect to the sorting order.
In particular, $A$ is Koszul and a normal Cohen-Macaulay domain.
\end{Theorem}

\begin{proof}
The proof is identical to \cite[Theorem 6.16]{Ene-Herzog}.
\end{proof}

\section{Toric rings attached to minimal vertex covers of proper interval graphs}

In this section we show that the vertex cover ideal of any proper interval graph is sortable and as an application of Theorem~\ref{Gbasis} we conclude that the toric rings attached to vertex cover ideals of such graphs are Koszul and normal Cohen-Macaulay domains. First we recall some definitions.

Let $G$ be a finite simple graph on the vertex set $V(G)=[n]$ and with the edge set $E(G)$. Recall that a subset $C\subseteq [n]$ is called a \emph{vertex cover} of $G$,
if it intersects any edge of $G$. Moreover, $C$ is called a \emph{minimal vertex cover} of $G$, if it is a vertex cover and no proper subset of $C$ is a vertex cover of $G$.
Let $C_1,\ldots,C_m$  be the minimal vertex covers of $G$. The {\em vertex cover ideal} or shortly cover ideal of $G$ is a squarefree monomial ideal in the polynomial ring $S=K[x_1,\ldots,x_n]$ defined as
$$I_G=(x_{C_1},\ldots,x_{C_m}),$$ where $x_{C_j}=\prod_{i\in C_j} x_i$. For a vertex $i\in V(G)$, set $N_G(i)=\{j\in V(G):\ \{i,j\}\in E(G)\}$  and $N_G[i]=N_G(i)\cup\{i\}$.

A graph $G$ is called a {\em proper interval} graph if there exists a labeling $[n]$ on its vertex set such that for all $i<j$, $\{i, j\}\in E(G)$ implies that the induced subgraph of $G$ on $\{i, i+1, \dots ,j\}$ is a complete subgraph of $G$. Proper interval graphs are well studied in the literature, see for example \cite{BW,G}. A nice characterization of these graphs in terms of the sortability of their independence complexes is given in~\cite{HKMR}.

We say that a monomial ideal $I$ is {\em sortable}, if $\mathcal{G}(I)$ is a sortable set of monomials, where  $\mathcal{G}(I)$ is the unique set of minimal monomial generators of $I$.

\begin{Theorem}\label{Properinterval}
Let $G$ be a proper interval graph on $[n]$ and $I_G$ be the cover ideal of $G$.
Then $I_G$ is a sortable ideal.
\end{Theorem}

\begin{proof}
Let $C_1$ and $C_2$ be minimal vertex covers of $G$ and set $u_i=x_{C_i}$ for $i=1,2$.
Let $\sort(u_1,u_2)=(v_1,v_2)$. First we show that $v_1,v_2\in I_G$.
Let $u_1u_2=x_{i_1}x_{i_2}\cdots x_{i_k}$ with $i_1\leq i_2\leq \cdots \leq i_k$. Then $v_1=\prod_{1\leq 2\ell+1\leq k} x_{i_{2\ell+1}}$ and $v_2=\prod_{2\leq 2\ell\leq k} x_{i_{2\ell}}$.
We set $D_1=\supp(v_1)$ and $D_2=\supp(v_2)$. By contradiction assume that $D_1$ is not a vertex cover of $G$.
Then there exists and edge $e$ of $G$ such that $e\cap D_1=\emptyset$. Since $e\cap \{i_1,i_2,\ldots,i_k\}\neq \emptyset$, we should have $e=\{i_{2\ell},i_{2t}\}$ for some $\ell$ and $t$ with $\ell<t\leq k/2$. Since $i_{2\ell}\leq i_{2\ell+1}\leq i_{2t}$, and $e\cap D_1=\emptyset$, we have $i_{2\ell}< i_{2\ell+1}< i_{2t}$. Since $G$  is proper interval, the induced subgraph of $G$ on the set $A=\{i_{2\ell},i_{2\ell+1},i_{2t}\}$ is the complete graph $K_3$. Hence $|A\cap C_1|\geq 2$ and $|A\cap C_2|\geq 2$. So there exists an element in $A$ which belongs to $C_1\cap C_2$. If $i_{2\ell}\in C_1\cap C_2$, then $i_{2\ell}=i_{2\ell-1}$ or $i_{2\ell}=i_{2\ell+1}$ which means $i_{2\ell}\in e\cap D_1$, a contradiction to our assumption. So $i_{2\ell}\notin C_1\cap C_2$. By a similar reason, $i_{2t}\notin C_1\cap C_2$. Thus $i_{2\ell+1}\in C_1\cap C_2$. Hence $i_{2\ell+1}=i_{2\ell+2}<i_{2t}$. So $2\ell+2<2t$ and then $i_{2\ell}< i_{2\ell+3}\leq i_{2t}$. If $i_{2\ell+3}=i_{2t}$, then $i_{2\ell+3}\in e\cap D_1$, a contradiction. Hence $i_{2\ell}<i_{2\ell+3}< i_{2t}$ and similar to the argument above, we obtain $i_{2\ell+3}=i_{2\ell+4}$. Then $2\ell+4<2t$. Proceeding this way, we get $2t=2\ell+(2t-2\ell)<2t$, a contradiction. Hence, $D_1$ is a vertex cover of $G$.
By a similar argument $D_2$ is a vertex cover of $G$.

Now, we show that $D_1$ is a minimal vertex cover of $G$. By contradiction assume that $D'_1=D_1\setminus \{i_{2r+1}\}$ is a vertex cover of $G$ for some $r\geq 0$. Then $N_G[i_{2r+1}]\subseteq D_1$. We show that either $N_G[i_{2r+1}]\subseteq C_1$ or $N_G[i_{2r+1}]\subseteq C_2$.  Since $G$ is a proper interval graph and $N_G[i_{2r+1}]\subseteq D_1$, we may write $N_G[i_{2r+1}]=[i_{2s+1},i_{2t+1}]$.
We set $A=\{i_{2\ell}:\ s+1\leq \ell\leq t\}$ and $B=\{i_{2\ell+1}:\ s\leq \ell\leq t\}$. Then $N_G[i_{2r+1}]=A\cup B\subseteq D_1$. This implies that $A\subseteq B$. Hence, $N_G[i_{2r+1}]=B$. Moreover, since $|B|=|A|+1$, we have $B=A\cup \{i_{2p+1}\}$ for some $s\leq p\leq t$.
Hence $x_A^2 x_{i_{2p+1}}=x_Ax_B=\prod_{j=2s+1}^{2t+1} x_{i_j}$. Thus $x_A^2 x_{i_{2p+1}}$ divides $x_{C_1}x_{C_2}$. Therefore, $x_A x_{i_{2p+1}}$ divides either
$x_{C_1}$ or $x_{C_2}$. This means that $B\subseteq C_1$ or $B\subseteq C_2$. Since $N_G[i_{2r+1}]=B$, we get $N_G[i_{2r+1}]\subseteq C_i$ for some $i\in\{1,2\}$. Thus
$C_i\setminus \{i_{2r+1}\}$ is a vertex cover of $G$, which contradicts to the minimality of $C_i$.
Hence $D_1$ is a minimal vertex cover. By a similar argument $D_2$ is a minimal vertex cover of $G$. Hence, $v_1,v_2\in \mathcal{G}(I_G)$.
\end{proof}

\medskip

For a graph $G$, the vertex cover algebra of $G$ is defined as the toric ring $$K[x_Ct:\  C\  \textrm{is a minimal vertex cover of $G$}].$$

By Theorems  \ref{Gbasis} and \ref{Properinterval} we obtain

\begin{Corollary}
Let $G$ be a proper interval graph. Then the vertex cover algebra of $G$ is Koszul and a normal Cohen-Macaulay domain.
\end{Corollary}

\section{Cover ideals with componentwise linear powers}

Let $G$ be a graph on $V(G)=\{x_1,\ldots,x_n\}$. A {\em clique partition} $\pi$ of $G$ is a partition of $V(G)$ into disjoint subsets $V_1,\ldots,V_m$ such that for each $1\leq i\leq m$ the induced subgraph of $G$ on $V_i$ is a complete graph. For a given clique partition $\pi=(V_1,\ldots,V_m)$ of $G$ and positive integers $r_1,\ldots,r_m$, the {\em clique multi-whiskered graph} $G^{\pi}(r_1,\ldots,r_m)$ is defined to have the vertex set $V(G)\cup(\bigcup_{i=1}^m\{z_1^{(i)},\ldots,z_{r_i}^{(i)}\})$
and the edge set $$E(G)\cup \bigcup_{i=1}^m (\bigcup_{x_j\in V_i} \{\{x_j,z_1^{(i)}\},\{x_j,z_2^{(i)}\},\ldots,\{x_j,z_{r_i}^{(i)}\}\}).$$

\begin{Example}
For the graph $G$ depicted in Figure \ref{figure1}, with the clique partition $\pi=(\{x_1,x_2,x_3\},\{x_4\},\{x_5,x_6\})$,  the graph $G^{\pi}(2,3,2)$ is shown below.

\begin{figure}[h]
\begin{center}
\includegraphics[height=4cm]{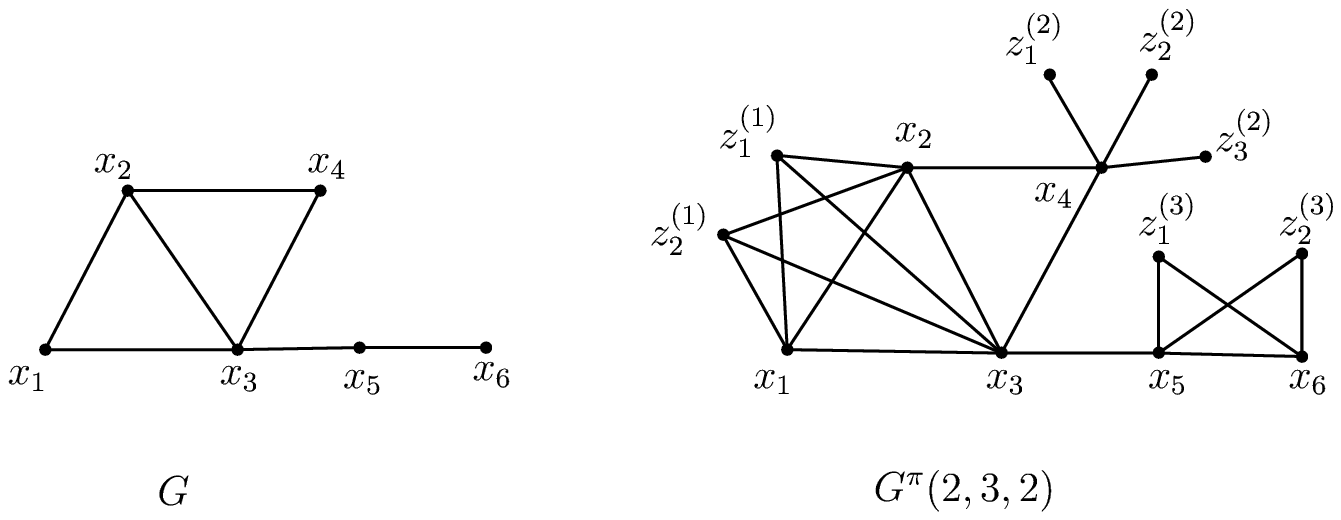}
\end{center}
\caption{}
\label{figure1}
\end{figure}
\end{Example}

It is easy to see that each minimal vertex cover of $G^{\pi}(r_1,\ldots,r_m)$ is of the form
$$C'=C\cup (\bigcup_{V_i\nsubseteq C} \{z_1^{(i)},\ldots,z_{r_i}^{(i)}\}),$$
where $C$ is a vertex cover (not necessarily minimal) of $G$ and vice versa.
Let $C_1,\ldots,C_q$ be all the vertex covers of $G$. Then $\{C'_1,\ldots,C'_q\}$ is the set of minimal vertex covers of $G^{\pi}(r_1,\ldots,r_m)$.
Let $$S=K[x_1,\ldots,x_n,z_1^{(1)},\ldots,z_{r_1}^{(1)},\ldots,z_1^{(m)},\ldots,z_{r_m}^{(m)}]$$ be the polynomial ring in $n+\sum_{i=1}^m r_i$ variables over a field $K$. Let $<_{\lex}$ denote the pure lexicographic order on $S$ induced by the ordering
$$z_1^{(1)}>\cdots>z_{r_1}^{(1)}> z_1^{(2)}> z_2^{(2)}>\cdots>z_{r_m}^{(m)}>x_1>\cdots>x_n$$
of the variables.
We associate each $C'_{\ell}$, $1\leq \ell\leq q$ with a squarefree monomial $u_{C'_{\ell}}=(\prod\limits_{x_j\in C'_{\ell}} x_j)(\prod\limits_{z_k^{(i)}\in C'_{\ell}} z_k^{(i)})\in S$.
One can assume that $u_{C'_1}>_{\lex} \cdots >_{\lex} u_{C'_q}$. The vertex cover ideal of $G^{\pi}(r_1,\ldots,r_m)$ is the squarefree monomial ideal
$I_{G^{\pi}(r_1,\ldots,r_m)}=(u_{C'_1},\ldots,u_{C'_q})$ and its Rees algebra is the toric ring $$\mathcal{R}(I_{G^{\pi}(r_1,\ldots,r_m)})=K[u_{C'_1}t,\ldots,u_{C'_q}t,x_1,\ldots,x_n,z_1^{(1)},\ldots,z_{r_m}^{(m)}]\subset S[t].$$

Let $T=S[y_1,\ldots,y_q]$ denote the polynomial ring and define the surjective map $\pi: T\rightarrow \mathcal{R}(I_{G^{\pi}(r_1,\ldots,r_m)})$ by setting
$\pi(x_j)=x_j$ for all $j$, $\pi(z_k^{(i)})=z_k^{(i)}$ for all $k$ and $i$, and $\pi(y_{\ell})=u_{C'_{\ell}}t$ for all $\ell$.
The toric ideal $J_{G^{\pi}(r_1,\ldots,r_m)} \subset T$ of $\mathcal{R}(I_{G^{\pi}(r_1,\ldots,r_m)})$ is the kernel of $\pi$.

Let $<'$ denote the pure lexicographic order on $K[y_1,\ldots,y_q]$ induced by the ordering
$y_1>\cdots>y_q$.

Now, we introduce the monomial order $<$ on $T$ for which
\begin{eqnarray}
\label{monomialorder}
(\prod x_i^{a_i})(\prod {z_k^{(i)}}^{b_{k,i}})(\prod y_i^{b_i}) &<& (\prod x_i^{a'_i})(\prod {z_k^{(i)}}^{b'_{k,i}})(\prod  y_i^{b'_i}),
\end{eqnarray}
if
\[\prod  y_i^{b_i} <'\prod  y_i^{b'_i}\ \  \text{or}\] 

\[\prod  y_i^{b_i}=\prod  y_i^{b'_i}\ \   \text{and} \ \  (\prod x_i^{a_i})(\prod {z_k^{(i)}}^{b_{k,i}}) <_{\lex} (\prod x_i^{a'_i})(\prod {z_k^{(i)}}^{b'_{k,i}}).
\]

\begin{Lemma}\label{x condition}
The ideal $J_{G^{\pi}(r_1,\ldots,r_m)}$ satisfies the $x$-condition with respect to $<$.
\end{Lemma}

\begin{proof}
Let $J'\subset K[y_1,\ldots,y_m]$ denote the toric ideal of $K[u_{C'_1}t,\ldots,u_{C'_q}t,]$ and $\mathcal{G}$ the reduced Gr\"{o}bner basis of $J'$ with respect to $<'$. Let $C_a$ be a vertex cover of $G$ and $1\leq i\leq m$ be an integer with $V_i\nsubseteq C_a$. Choose $x_j\in V_i\setminus C_a$ and let $C_b=C_a\cup\{x_j\}$. Then $x_ju_{C'_a}=z_1^{(i)}\cdots z_{r_i}^{(i)}u_{C'_b}$ and $u_{C'_b}<_{\lex}u_{C'_a}$. Hence $f_{j,a}=x_jy_a-z_1^{(i)}\cdots z_{r_i}^{(i)} y_b\in J_{G^{\pi}(r_1,\ldots,r_m)}$ with $\ini_<(f_{j,a})=x_jy_a$. Let $\mathcal{G'}$ denote the set of such binomials $f_{j,a}$.

We claim that $\mathcal{G}\cup \mathcal{G'}$ is a Gr\"{o}bner basis of $J_{G^{\pi}(r_1,\ldots,r_m)}$ with respect to $<$.
Let $\Omega\subset T$ denote the set of those monomials $w\in T$ for which none of the initial monomials $\ini_<(g)$ with $g\in \mathcal{G'}$ and $\ini_{<'}(h)$ with $h\in \mathcal{G}$ divides $w$. Using~\cite[Lemma 1.1]{AHH}, we need to show that the monomials in $\Omega$ are linearly independent in $T/J_{G^{\pi}(r_1,\ldots,r_m)}$. In other words,
if $w,w'\in \Omega$ with $w\neq w'$, then $\pi(w)\neq \pi(w')$.

Write $w=w_xw_zw_y$, where $w_x$ is a monomial in $K[x_1,\ldots,x_n]$, $w_z$ is a monomial in $K[z_k^{(i)},\ 1\leq i\leq m ,\ 1\leq k\leq r_i]$ and $w_y$ is a monomial in $K[y_1,\ldots,y_q]$. Suppose $w$ and $w'$ possess no common variable with $\pi(w)=\pi(w')$. Then one has $\deg(w_y)=\deg(w'_y)$.

Let $w_x\neq 1$ and $x_j|w_x$. Then $x_j\nmid w'_x$. Since $\deg(w_y)=\deg(w'_y)$, it follows that there exists $1\leq \ell\leq q$ for which $y_{\ell}|w_y$ and $x_j\notin C_{\ell}$. Since $x_jy_{\ell}|w$ and $x_jy_{\ell}=\ini_<(f_{j,\ell})$, it follows that $w\notin \Omega$, a contradiction. Hence $w_x=w'_x=1$.

Let $w_z\neq 1$ and $z_k^{(i)}|w_z$. Let $d$ denote the number of $y_t$ for which $y_t|w_y$ and $z_k^{(i)}\in C'_t$ and $d'$ denote the number of $y_{t'}$ for which $y_{t'}|w'_y$ and $z_k^{(i)}\in C'_{t'}$. Then $d<d'$. Let $V_i=\{x_{i_1},\ldots,x_{i_s}\}$, and let $a_{\ell}\geq 0$ be the biggest integer for which $x_{i_{\ell}}^{a_{\ell}}|\pi(w_y)$ and
$a'_{\ell}\geq 0$ that for which $x_{i_{\ell}}^{a'_{\ell}}|\pi(w'_y)$. Then $\sum_{\ell=1}^s a_{\ell}>\sum_{\ell=1}^s a'_{\ell}$. This contradicts to $\pi(w)=\pi(w')$. Hence $w_z=w'_z=1$.

Therefore $w=w_y$ and $w'=w'_y$. Since $\mathcal{G}$ is the reduced Gr\"{o}bner basis of $J'$ with respect to $<'$ and since $w,w'\in \Omega$, we should have $w=w'$.
\end{proof}

\begin{Theorem}\label{powers}
Let $G$ be a graph, $\pi$ be an arbitrary clique partition of $G$ (to $m$ parts) and $r_1,\ldots,r_m$ be positive integers. Then all powers of the vertex cover ideal $I_{G^{\pi}(r_1,\ldots,r_m)}$ have linear quotients and hence are componentwise linear.
\end{Theorem}

\begin{proof}
By Lemma~\ref{x condition}, $J_{G^{\pi}(r_1,\ldots,r_m)}$ satisfies the $x$-condition with respect to $<$. Hence by the proof of \cite[Theorem 2.3]{HHM}, for any positive integer $k$, the ideal $(I_{G^{\pi}(r_1,\ldots,r_m)})^k$  has a system of generators $h_1,\ldots,h_s$ (not necessarily minimal), which possess an order of linear quotients $h_1<\cdots<h_s$ and that each of these generators is of the form $h_i=u_{C'_{i_1}}\cdots u_{C'_{i_k}}$  such that $h_i^*=y_{i_1}\cdots y_{i_k}$ is a standard monomial of $T$ with respect to $<$.

We claim that $\{h_1,\ldots,h_s\}$ is indeed the minimal generating set of $(I_{G^{\pi}(r_1,\ldots,r_m)})^k$. To this aim it is enough to show that if $h_j=wh_i$ for some integers $i$ and $j$ and some monomial $w\in S$, then $w=1$ and $i=j$.

Let $h_i=u_{C'_{i_1}}\cdots u_{C'_{i_k}}$ and $h_j=u_{C'_{j_1}}\cdots u_{C'_{j_k}}$. Write $w=w_xw_z$, where $w_x$ is a monomial in $K[x_1,\ldots,x_n]$ and $w_z$ is a monomial in $K[z_k^{(i)},\ 1\leq i\leq m ,\ 1\leq k\leq r_i]$.
Let $w_z\neq 1$ and $z_k^{(p)}|w_z$. Let $d$ denote the number of integers $t$ for which $z_k^{(p)}\in C'_{i_t}$ and $d'$ denote the number of integers $t'$ for which  $z_k^{(p)}\in C'_{j_{t'}}$. Then $d<d'$. Let $V_p=\{x_{p_1},\ldots,x_{p_s}\}$, and let $a_{\ell}\geq 0$ be the biggest integer for which $x_{p_{\ell}}^{a_{\ell}}|h_i$ and
$a'_{\ell}\geq 0$ that for which $x_{p_{\ell}}^{a'_{\ell}}|h_j$. Then $\sum_{\ell=1}^s a_{\ell}>\sum_{\ell=1}^s a'_{\ell}$. This contradicts to $h_j=wh_i$. Hence $w_z=1$.

Let $w_x\neq 1$ and $x_t|w$ for some integer $t$. Then $x_t\notin C_{i_{a}}$
for some $1\leq a\leq k$. Thus for $C_b=C_{i_a}\cup\{x_t\}$ we have $x_tu_{C'_{i_a}}=z_1^{(t)}\cdots z_{r_t}^{(t)}u_{C'_b}$.
Since $x_t| w$ and $u_{C'_{i_a}}|h_i$, we may write $wh_i=w'h'$, where $w'=(w/x_t)z_1^{(t)}\cdots z_{r_t}^{(t)}$ and $h'=(h_i/u_{C'_{i_a}})u_{C'_b}$. Therefore
$h_j=w'h'$ with $w'_z\neq 1$. This is impossible similar to the discussion in the previous paragraph. Thus $w_x=1$.
So we have $h_j=h_i$. If $i\neq j$, then this implies that $h_j^*-h_i^*\in J_{G^{\pi}(r_1,\ldots,r_m)}$. So either $h_i^*$ or $h_j^*$ belongs to $\ini_<(J_{G^{\pi}(r_1,\ldots,r_m)})$, a contradiction. Hence $i=j$, as desired.
\end{proof}

\end{document}